%% file: main.tex
\documentclass[11pt]{article}
\usepackage[letterpaper,margin=0.95in,headheight=15pt]{geometry}
\usepackage[T1]{fontenc}
\usepackage{lmodern,amsmath,amssymb,amsthm,mathtools,microtype}
\usepackage{enumitem,xcolor,fancyhdr}
\definecolor{navy}{RGB}{25,55,85}
\usepackage[colorlinks=true,linkcolor=navy,citecolor=navy,urlcolor=navy]{hyperref}
\newcommand{\E}{\mathbb E}
\newcommand{\Pp}{\mathbb P}

\newcommand{\one}{\mathbf 1}
\newcommand{\Ucal}{\mathcal U}
\newcommand{\Law}{\mathcal L}

\newcommand{\wh}{\widehat}
\newtheorem{theorem}{Theorem}[section]
\newtheorem{lemma}[theorem]{Lemma}
\newtheorem{proposition}[theorem]{Proposition}
\newtheorem{corollary}[theorem]{Corollary}
\theoremstyle{definition}
\newtheorem{definition}[theorem]{Definition}
\theoremstyle{remark}
\newtheorem{remark}[theorem]{Remark}
\numberwithin{equation}{section}
\setlist{itemsep=3pt,topsep=5pt,leftmargin=*}
\begin{document}

\title{A Note on a Result of Chen Li: the Thinning Problem via the Kernel Method} 
\author{Xuan Fang\footnote{26110180010@m.fudan.edu.cn} \quad and \quad Tianyu Wang\footnote{wangtianyu@fudan.edu.cn}}
\date{Sep 12, 2026} 
\maketitle 

\begin{abstract} 
Recently, C. Li posted a preprint that resolves the fractional version of Talagrand's discrete convexity problem. In this note, we use the method of C. Li to answer another problem of Talagrand. 




\end{abstract}

\input{intro}

\subsection{Definitions and the Main Result}\label{sec:main}
We identify a subset of $[N]=\{1,\ldots,N\}$ with its indicator vector in $\{0,1\}^N$. For $0<s<1$, let $\mu_s$ be the product measure 
\[
\mu_s(A)=s^{|A|}(1-s)^{N-|A|}\qquad(A\subseteq[N]).
\]
A family $\Ucal\subseteq2^{[N]}$ is \emph{increasing} if $A\in\Ucal$ and $A\subseteq B$ imply $B\in\Ucal$. For any random variable $W$, we use $\Law(W)$ to denote its distribution law. Next we recall some conventions in the following definition. 

\begin{definition}[Spreadness and independent thinning]
A probability measure $\nu$ is $q$-spread if a random set $J\sim\nu$ satisfies
\begin{equation}\label{eq:spread}
\Pp(I\subseteq J)\le q^{|I|}\qquad(I\subseteq[N]).
\end{equation}
For $t \in (0,1)$, define the retention law $T_t\nu=\Law(J\cap X_t)$, where $J \sim \nu$ and $X_t\sim\mu_t$ is independent of $J$. 
\end{definition}

We use a small spread parameter: the alternative $R$-spread convention with bound $R^{-|I|}$ corresponds to $q=1/R$.

For laws $\rho,\sigma$ on the Hamming cube, write $\rho\le_d\sigma$ if there exists a coupling $(A,B)$ with marginals $\rho$ and $\sigma$ respectively, such that $A\subseteq B$ almost surely. This is equivalent to
\begin{equation}\label{eq:upsetorder} 
\rho(\Ucal)\le\sigma(\Ucal)\qquad\text{for every increasing family }\Ucal.
\end{equation}

\begin{theorem}[Domination after thinning]\label{thm:main}
Let $0<t,q,p<1$. If $\nu$ is $q$-spread and
\begin{equation}\label{eq:condition}
\eta:=\frac{tq(1-p)}{p(1-t)}\le1,
\end{equation}
then $T_t\nu\le_d\mu_p$. In particular, for every $q$-spread law,
\begin{equation}\label{eq:generaltheorem}
\displaystyle T_t\nu\le_d\mu_{p_*},\qquad
p_*:=\frac{tq}{1-t+tq}.
\end{equation}
The assertion is uniform in the dimension $N$.
\end{theorem}
Theorem \ref{thm:main} is established by means of Li's argument and therefore settles Conjecture 7.8 of \cite{talagrand}, stated below as a corollary.
\begin{corollary}\cite[Conjecture 7.8]{talagrand}\label{cor:conjecture}
Fix
\[
0<\alpha\le\frac{\sqrt5-1}{2}.
\]
For every $N\ge1$, every $p\in(0,1)$, and every $\alpha p$-spread random set $J$, independent thinning gives
\begin{equation}\label{eq:conjecture}
\displaystyle J\cap X_\alpha\le_d X_p.
\end{equation}
Here $X_\alpha\sim\mu_\alpha$ is independent of $J$, and $X_p$ denotes a random set of law $\mu_p$ in a containing coupling. In particular, $\alpha=1/2$ is a universal choice.
\end{corollary} 

\begin{remark}
    As remarked in \cite{talagrand}, a positive answer to Conjecture 7.8 in \cite{talagrand} would also give positive answers to Conjectures 7.3 and 7.11 in \cite{talagrand}. 
\end{remark}

This is the thinning conjecture in Talagrand~\cite{talagrand}.
The argument below adapts the vanishing-kernel method in Li~\cite{li}. Li uses a kernel that vanishes when a parent cannot be covered by the union of two sets, whereas in this note, the kernel vanishes when a thinned parent is not contained in a target set. 

\section{The Kernel Argument}\label{sec:kernel}

To start with, we recall some Fourier analysis basics for Boolean functions; See \cite{o2014analysis} for a general reference; See \cite{friedgut2008,ellis_filmus_friedgut2012} for spectral methods to intersecting families, and \cite{haemers2021} for spectral methods to Hoffman's ratio bound, both from an analytic perspective. 
For $s\in(0,1)$ and $x \in \{ 0,1 \}$, define
\[
\phi_s(x)=\frac{x-s}{\sqrt{s(1-s)}},\qquad
\omega_I^{(s)}(x)=\prod_{i\in I}\phi_s(x_i),\qquad \omega_\varnothing^{(s)}=1.
\]
The one-coordinate identities $\E_{\mu_s}\phi_s=0$ and $\E_{\mu_s}\phi_s^2=1$, together with independence, show that the $2^N$ functions $\omega_I^{(s)}$ are an orthonormal basis. Consequently, every real function $f$ on the cube satisfies
\begin{equation}\label{eq:walsh}
f=\sum_{I\subseteq[N]}\wh f_s(I)\omega_I^{(s)},\qquad
\wh f_s(I)=\E_{\mu_s}[f\omega_I^{(s)}],\qquad
\E_{\mu_s}f^2=\sum_I\wh f_s(I)^2.
\end{equation}

Throughout this section, $t$ is the retention probability (in the operator $T_t$) and $p$ is the target product density. Set
\begin{equation}\label{eq:beta}
\beta=\sqrt{\frac{t(1-p)}{p(1-t)}}.
\end{equation}

Inspired by the argument of Li, we define the following kernel. 
\begin{lemma}[The one-coordinate kernel]\label{lem:kernel}
Define
\begin{equation}\label{eq:h}
h(x,w)=1+\beta\phi_t(x)\phi_p(w)
      =1+\frac{(x-t)(w-p)}{p(1-t)}.
\end{equation}
Then
\begin{equation}\label{eq:htable}
\bigl(h(x,w)\bigr)_{x,w\in\{0,1\}}
=
\begin{pmatrix}
\dfrac1{1-t}&\dfrac{p-t}{p(1-t)}\\[6pt]
0&\dfrac1p
\end{pmatrix}.
\end{equation}
In particular,
\begin{equation}\label{eq:kernelproperties}
h(1,0)=0,\qquad
\E_{W\sim\mathrm{Bernoulli}(p)}h(x,W)=1\quad \text{ for $x=0,1$}.
\end{equation}
\end{lemma}
\begin{proof}
Substitution of $x,w\in\{0,1\}$ into~\eqref{eq:h} gives the displayed matrix. The normalization follows from $\E\phi_p(W)=0$.
\end{proof}

For a deterministic set $j\subseteq[N]$, define
\begin{equation}\label{eq:kandL}
k_j(x,w)=\prod_{i\in j}h(x_i,w_i),\qquad
(L_jf)(x)=\E_{W\sim\mu_p}[k_j(x,W)f(W)].
\end{equation}
Coordinates outside $j$ contribute the constant factor $1$. The empty product equals $1$, so $L_\varnothing f$ is the constant $\E_{\mu_p}f$.

\begin{lemma}[Normalization and noncontainment]\label{lem:containment}
For every $j,x,w\subseteq[N]$,
\begin{equation}\label{eq:zero}
j\cap x\nsubseteq w\quad\Longrightarrow\quad k_j(x,w)=0.
\end{equation}
Moreover, $L_j1=1$ for every $j$.
\end{lemma}
\begin{proof}
If $j\cap x\nsubseteq w$, choose $i\in(j\cap x)\setminus w$. The corresponding factor in~\eqref{eq:kandL} is $h(1,0)=0$. For the second assertion, independence of the coordinates of $W$ and~\eqref{eq:kernelproperties} give
\[
(L_j1)(x)=\prod_{i\in j}\E h(x_i,W_i)=1.
\]
\end{proof}

\begin{remark}[The kernel may be signed]\label{rem:signed}
When $p<t$, the entry $h(0,1)$ is negative. Neither $k_j$ nor $L_j$ is assumed to be positive. They are auxiliary algebraic objects. The proof uses the zero pattern and normalization of Lemma~\ref{lem:containment}, and then takes the square of the real number $L_jf(x)$. No probabilistic coupling is inferred from the kernel itself.
\end{remark}

\section{The square inequality and the positive certificate}\label{sec:certificate}
\begin{lemma}[Pointwise square bound]\label{lem:square}
Let $\Ucal$ be increasing and $f=\one_\Ucal$. For every deterministic $j,x\subseteq[N]$,
\begin{equation}\label{eq:square}
\displaystyle f(j\cap x)\le (L_jf(x))^2.
\end{equation}
\end{lemma}
\begin{proof}
Suppose first that $j\cap x\in\Ucal$. If $k_j(x,w)\ne0$, Lemma~\ref{lem:containment} implies $j\cap x\subseteq w$. Since $\Ucal$ is increasing, $f(w)=1$. Thus, for every $w$,
\[
k_j(x,w)f(w)=k_j(x,w).
\]
Averaging under $\mu_p$ gives $L_jf(x)=L_j1(x)=1$. Hence both sides of~\eqref{eq:square} equal $1$. If $j\cap x\notin\Ucal$, the left-hand side is zero and the right-hand side is a square, so the inequality again holds.
\end{proof}

\begin{lemma}[Walsh expansion of the lifted function]\label{lem:expansion}
For any real function $f$ on the cube and every deterministic $j$,
\begin{equation}\label{eq:liftedexpansion}
L_jf(x)=\sum_{I\subseteq j}\beta^{|I|}\wh f_p(I)\omega_I^{(t)}(x).
\end{equation}
Consequently,
\begin{equation}\label{eq:liftedparseval}
\E_{X_t\sim\mu_t}(L_jf(X_t))^2
=\sum_{I\subseteq j}\beta^{2|I|}\wh f_p(I)^2.
\end{equation}
\end{lemma}
\begin{proof}
Expand the finite product in~\eqref{eq:kandL}:
\[
k_j(x,w)=\prod_{i\in j}\bigl(1+\beta\phi_t(x_i)\phi_p(w_i)\bigr)
=\sum_{I\subseteq j}\beta^{|I|}\omega_I^{(t)}(x)\omega_I^{(p)}(w).
\]
Multiply by $f(w)$ and average in $w$ under $\mu_p$ to obtain~\eqref{eq:liftedexpansion}. Orthogonality under $\mu_t$ gives~\eqref{eq:liftedparseval}.
\end{proof}

\begin{proposition}[A positive majorant on the parent cube]\label{prop:certificate}
Let $\Ucal$ be increasing and $f=\one_\Ucal$. Define
\begin{equation}\label{eq:coefficients}
c_I=\left(\frac{t(1-p)}{p(1-t)}\right)^{|I|}\wh f_p(I)^2
\qquad(I\subseteq[N]).
\end{equation}
Then every $c_I$ is nonnegative, and for every deterministic parent $j$,
\begin{equation}\label{eq:positivecertificate}
\displaystyle
\Pp(j\cap X_t\in\Ucal)\le\sum_{I\subseteq j}c_I.
\end{equation}
The coefficients depend on $\Ucal,t,p$ and are independent of the parent law $\nu$.
\end{proposition}
\begin{proof}
    Average the square bound~\eqref{eq:square} over $X_t\sim\mu_t$ and plugging Lemma \ref{lem:expansion} yields
    \begin{align*}
		\Pp(j\cap X_t\in\Ucal)
		&=\E f(j\cap X_t)\\
		&\le \E\bigl(L_jf(X_t)\bigr)^2\\
        &=\sum_{I,K\subseteq j}
		\beta^{|I|+|K|}\wh f_p(I)\wh f_p(K)
		\E\bigl[\omega_I^{(t)}(X_t)\omega_K^{(t)}(X_t)\bigr]\\
		&=\sum_{I\subseteq j}\beta^{2|I|}\wh f_p(I)^2
		=\sum_{I\subseteq j}
		\left(\frac{t(1-p)}{p(1-t)}\right)^{|I|}
		\wh f_p(I)^2
		=\sum_{I\subseteq j}c_I.
	\end{align*}
\end{proof}

\section{Proof of domination and of Theorem \ref{thm:main}}\label{sec:proof}
\begin{proof}[Proof of Theorem~\ref{thm:main}]
Fix an increasing family $\Ucal$ and put $f=\one_\Ucal$. Let $J\sim\nu$ and let $X_t\sim\mu_t$ be independent. Conditional on $J=j$, the mask still has law $\mu_t$. Thus Proposition~\ref{prop:certificate} gives
\begin{align}
\Pp(J\cap X_t\in\Ucal)
&\le\E_\nu\sum_{I\subseteq J}\beta^{2|I|}\wh f_p(I)^2\nonumber\\
&=\sum_{I\subseteq[N]}\beta^{2|I|}\wh f_p(I)^2\Pp(I\subseteq J)\nonumber\\
&\le\sum_{I\subseteq[N]}(\beta^2q)^{|I|}\wh f_p(I)^2\label{eq:averaging}\\
&\le\sum_{I\subseteq[N]}\wh f_p(I)^2\nonumber\\
&=\E_{\mu_p}f^2=\mu_p(\Ucal).\nonumber
\end{align}
The first inequality uses the independent thinning, the next uses spreadness with nonnegative coefficients, and the last inequality uses $\beta^2q=\eta\le1$. Parseval and the Boolean identity $f^2=f$ give the final equality. This proves $T_t\nu\leq_d\mu_p$ under condition \eqref{eq:condition}.

Finally, condition~\eqref{eq:condition} is equivalent to
\[
tq(1-p)\le p(1-t)
\quad\Longleftrightarrow\quad
p\ge\frac{tq}{1-t+tq}.
\]
The latter fraction lies in $(0,1)$, so choosing equality proves~\eqref{eq:generaltheorem}.
\end{proof}

\begin{proof}[Proof of Corollary~\ref{cor:conjecture}]
Use Theorem~\ref{thm:main} with $t=\alpha$ and $q=\alpha p$. Then
\[
\eta=\frac{\alpha^2(1-p)}{1-\alpha}.
\]
If $\alpha\le(\sqrt5-1)/2$, then $\alpha^2\le1-\alpha$, so $\eta\le1$ for every $p\in(0,1)$. The theorem gives $T_\alpha\nu\le_d\mu_p$ in every dimension.
\end{proof}

The independence required by the conjecture is the independence of $J$ and the mask in the definition of $T_\alpha\nu$. The containing target produced by the coupling argument may depend on the source. If the original pair $(J,X_\alpha)$ is to be retained, sample the target conditionally on $Y=J\cap X_\alpha$ using the conditional probabilities of the containing coupling. This preserves the law of $(J,X_\alpha)$ and gives $Y\subseteq X_p$ almost surely.

\section{Summary} 

This is a note that expands the scope of a recent argument of C. Li, and answers Conjecture 7.8 (thus also Conjectures 7.3 and 7.11) of \cite{talagrand}.

\section*{Statement of AI Usage and Human Effort}
The proof was developed by GPT6 under the authors' guidance and was human-checked by the authors. More specifically, the authors used Li's method and proposed a kernel function identifying the event "$J\cap X_\alpha\subseteq X_p$" (while Li's original kernel function identifies the event "$S\in\mathcal F^{(2)}$"), which is exactly the kernel function used in the paper. GPT6 used the human-proposed kernel function and proves the thinning conjecture. The authors take full responsibility of the content in the paper. 



\bibliographystyle{plain}
	\bibliography{references}
\end{document}

%% file: intro.tex
\section{Introduction} 

In \cite{talagrand}, Talagrand proposed a collection of problems at the intersection of probability and combinatorics. This note of Talagrand has simulated several breakthrough results \cite{park_pham2024_kahn_kalai,park_pham2024_selectors,pham_sharp_selectors,hua_song_tudose2026}, ranging from smallness of the witness of selector processes and Gaussian decomposition of sub-Gaussian random variables. 



The central discrete problem in \cite{talagrand} is Conjecture~7.1: a constant number of unions of a large family should leave a exceptional class with a small witness class. Around this central problem, there is a list of problems concerning explicit covers, spread probability measures, and the effect of taking unions. In particular, studies on small covers for positive stochastic processes
\cite{bednorz2022suprema,bednorz2024small}
and a reformulation of Talagrand's discrete convexity conjecture
\cite{ascoli_he_park_talagrand2026} provide further motivations for this problem. 

A fractional version of Conjecture~7.1 in \cite{talagrand} is Conjecture 7.2 in \cite{talagrand}, which asks whether a constant number of unions of a large family should leave an exceptional class with a weakly small witness class. This problem is closely linked to the sunflower problem posed
by Erd\H{o}s and Rado~\cite{erdos_rado1960}, which asks how large a family of $k$-element sets must be to contain $r$ distinct members whose pairwise intersections all equal a common core. Recent progress on this problem has highlighted the role of spread measures, linking sunflower bounds \cite{alweiss_lovett_wu_zhang2021,2020Coding,bell2021note,stoeckl2022lecture,frankston_kahn_narayanan_park2021} to fractional expectation thresholds and fractional covers~\cite{kahn2007thresholds,frankston_kahn_narayanan_park2021,frankston2022problem,dubroff2024note}.
Several days ago, C. Li posted a preprint \cite{li} that resolves this fractional discrete convexity problem of Talagrand. 


Building on the techniques of C. Li, this note resolves another problem posed by Talagrand. Specifically, we show that Li's argument also settles Conjecture 7.8 of \cite{talagrand}, which asserts that every spread measure is, up to averaging over subsets, stochastically dominated by a product measure. This consequently resolves Conjectures 7.3 and 7.11 of \cite{talagrand}.